\documentclass[11pt]{amsart}
\usepackage{amsmath}
\usepackage{amsfonts}
\usepackage{amsthm}
\usepackage{amssymb}
\usepackage{mathrsfs}
\usepackage{latexsym}
\usepackage{verbatim}
\usepackage{diagrams}
\usepackage{a4wide}
\usepackage{enumerate}
\usepackage{hyperref}

\renewcommand{\thetheoremName}

\def\Iw{\mathrm{Iw}}
\def\GL{\mathrm{GL}}
\def\vareps{\varepsilon}

\def\cris{\mathrm{cris}}
\def\asp{\mathrm{asp}}

\def\T{\mathbf{T}}

\def\Ee{\mathscr{E}}

\newarrow{Equals}{=}{=}{=}{=}{=}
\newarrow{pots}{.}{.}{.}{.}{>}
\def\m{\mathfrak{m}}
\def\T{\mathbf{T}}
\def\Z{\mathbf{Z}}

\def\OL{\mathcal{O}}

\def\Q{\mathbf{Q}}
\def\F{\mathbf{F}}

\def\PGL{\mathrm{PGL}}
\def\R{\mathbf{R}}
\def\C{\mathbf{C}}

\def\End{\mathrm{End}}

\def\GSp{\mathrm{GSp}}
\def\PGSp{\mathrm{PGSp}}

\def\Sp{\mathrm{Sp}}
\def\GL{\mathrm{GL}}
\def\rank{\mathrm{rank}}

\def\Qbar{\overline{\Q}}

\def\ad{\mathrm{ad}}
\def\SL{\mathrm{SL}}

\def\SU{\mathrm{SU}}

\def\eps{\epsilon}

\def\Frob{\mathrm{Frob}}
\DeclareMathOperator{\Spec}{\mathrm{Spec}}

\newtheorem{theorem}{Theorem}[section]

\newtheorem{lemma}[theorem]{Lemma}

\def\ord{\mathrm{ord}}

\def\rbar{\overline{r}}

\newcommand{\UZ}{U \kern-.1em{Z}}

\newcommand{\VF}{V \kern-.07em{F}}

\def\Sp{\mathrm{Sp}}

\begin{document}

\def\fnot#1{\footnote{#1}\marginpar{\arabic{footnote}}}
\newcommand{\mar}[1]{\marginpar{\tiny #1}}

\title{Bloch--Kato conjectures for automorphic motives} 
\author{Frank Calegari \and David Geraghty \and Michael Harris} 
\thanks{The first author was supported in part by  NSF  Grant
  DMS-1404620. The second author was
  supported in part by NSF Grants DMS-1200304 and DMS-1128155. The third author
  was supported in part by  NSF DMS-1404769. 
  M.H.'s research received funding from the European Research Council under the European Community's Seventh Framework Programme (FP7/2007-2013) / ERC Grant agreement no. 290766 (AAMOT)}

\maketitle
\tableofcontents

\section{Introduction}

Assume (for this paragraph only) the standard conjectures, 
and suppose that $M$ is a pure irreducible Grothendieck motive over $\Q$ with coefficients
in (say) a totally real field $E$.
We make no assumption on the regularity or self-duality of $M$. According to conjectures
of Hasse--Weil, Langlands, Clozel, and others, one expects that the motive $M$
is automorphic, and corresponds to an algebraic cuspidal automorphic representation $\pi$
for $\GL(n)/\Q$ such that $L(\pi,s) = L(M,s)$. 
By a theorem of Jacquet and Shalika~\cite{J}, the $L$-function
$$L(M \times M^{\vee},s) = L(\pi \times \pi^{\vee},s)$$
is meromorphic for $\mathrm{Re}(s) > 0$ and
has a simple pole at $s = 1$. Let $\ad^0(M)$ be the pure motive of weight
zero  with coefficients in $E$ such that $\ad^0(M) \oplus E =  M \times M^{\vee}$.
Then
$$L(\ad^0(M),s) = \frac{L(\pi \times \pi^{\vee},s)}{\zeta(s)},$$
and $L(\ad^0(M),1) \ne 0$ is finite.
According to conjectures of Deligne and Bloch--Kato~\cite{BK}, for any pure de Rham representation
$V$, there is an equality:
$$ \dim H^1_f(G_{\Q},V) - \dim H^0(G_{\Q},V) = \ord_{s=1} L(V^*,s).$$
In particular, if we take $V = V^* = \ad^0(M)$, then we expect that $H^1_f(G_{\Q},\ad^0(M))$ should
vanish.
This is a special case of the more general fact that $H^1_f(\Q,V)$ should be trivial for any $p$-adic
representation $V$ arising from a 
pure motive $M$ of weight $w \ge 0$.
One also conjectures that the value of
 $L(\ad^0(M),1)$, 
 after normalization by some suitable period should lie in $\Q^{\times}$.
Moreover, after equating $M$ with its \'{e}tale realization for some prime $p$, the normalized
$L$ function  should have the same valuation as the order of a corresponding Selmer group
$H^1_{f}(\Q,\ad^0(M) \otimes \Q_p/\Z_p)$.

\medskip

No longer assuming any conjectures,
suppose that $M = \{r_{\lambda}\}$ is now a weakly compatible system of $n$-dimensional
 irreducible Galois representations of $G_{\Q}$, and
suppose moreover that $M$ is automorphic,  that is, it  corresponds to a cuspidal form $\pi$ for
$\GL(n)/\Q$ in a manner compatible with the local Langlands correspondence. Then, even without the standard conjectures,  it makes
sense to ask, for a $p$-adic representation
$r: G_{\Q} \rightarrow \GL_n(\OL)$ coming from $M$ (for some finite extension  $K/\Q_p$ with
ring of integers $\OL$),  if the Selmer group
$$H^1_f(\Q,\ad^0(r) \otimes K/\OL)$$
is finite. Theorems of this
kind were first proved for $n = 2$ by Flach~\cite{Flach}, and
they are also closely related to modularity lifting theorems as proved by 
Wiles~\cite{W,TW}, see (in particular)~\cite{FDG}.  More precisely, the order of this group
is related to the order of a congruence ideal between modular forms. 
In this paper, we prove versions of these results for modular
abelian surfaces and (conditionally) compatible
families of~$n$-dimensional representations whose existence was only recently proved to exist~\cite{HLTT}. 
The main theorem is the following.

\begin{theorem} \label{theorem:main}
Let $A/\Q$ be a semistable modular abelian surface
with~$\End(A) = \Z$. Let $p$ be a prime such that:
\begin{enumerate}
\item $p$ is sufficiently large with respect to some constant depending only on~$A$.
\item $A$ is ordinary at~$p$, and if  $\alpha, \beta$ are the unit root eigenvalues of $D_{\cris}(V)$, then
$$(\alpha^2-1)(\beta^2-1)(\alpha  - \beta)(\alpha^2 \beta^2 - 1) \not\equiv 0 \mod p.$$
\end{enumerate}
Then 
$$H^1_f(\Q,\mathrm{asp}^0(r) \otimes \Q_p/\Z_p) = 0$$
 where $\asp^0(r)$ is the $10$-dimensional adjoint representation of
$\PGSp(4)$.
Moreover, the set of primes~$p$ satisfying these conditions has density one.
\end{theorem}

For the families of Galois representations
constructed in~\cite{HLTT}, we  prove the following result.

\begin{theorem}  \label{theorem:conditional}
Let $\pi$ be a weight zero regular algebraic cuspidal representation for $\GL(n)/F$
for a CM field $F$ and coefficients in~$E$. 
Let~$\lambda$ be a prime of~$\OL_E$ dividing~$p$, and let
$$r = r_{\lambda}(\pi): G_{F} \rightarrow \GL_n(\OL)$$
 be a $p$-adic representation associated to $\pi$ with determinant~$\eps^{n(1-n)/2}$.
Assume that
\begin{enumerate}
\item $\rbar |_{F(\zeta_p)}$ has enormous image in the sense of~\cite{CG}~\S9.2,
\item Let~$v \ne p$ be a prime at which~$\pi$
is ramified.
\begin{enumerate} 
\item \label{convenient}  $\pi_v$ is an unramified twist of the Steinberg representation.
\item \label{difficult} 
The representations 
$r |_{G_{I_v}}$  and~$\rbar |_{G_{I_v}}$ are  unipotent.
For a topological generator~$\sigma_v  \in I_v$ of tame inertia,
$\rbar(\sigma)$ consists of a single block, namely:
$$\dim  \ker(\rbar(\sigma) - \mathrm{id}_n) = 1.$$
\end{enumerate}
\item $p$ is sufficiently large with respect to some constant depending only on~$\pi$.
\end{enumerate}
Assume all of Conjecture~B of~\cite{CG} except assumption~$(4)$.
Then the Selmer group~$H^1_f(F,\ad^0(r) \otimes K/\OL)$ is  trivial. 
\end{theorem}

Note that Conjecture~B of~\cite{CG} consists of five parts:
The first part concern local--global compatibility at~$v|p$, which is still open.
The second and third parts concern local--global compatibility at finite~$v$.
Here there is work in characteristic zero by Varma~\cite{Varma}, although
arguments of this nature should also  apply
to the Galois representations constructed by Scholze~\cite{Scholze}, at least
for modularity lifting purposes (since for modularity lifting it is usually
sufficient to have local--global compatibility up to~$N$-semi-simplification). 
The fifth part is essentially addressed
in~\cite{CG}, and also (in a different and arguably superior manner) in~\cite{KT}. Hence the
main remaining issue is local--global compatibility at~$\ell = p$.

Unlike the case of Theorem~\ref{theorem:main}, we do not know whether
 Theorem~\ref{theorem:conditional} applies for infinitely many~$p$. One reason is that we do not
even know that the representations~$r_{\lambda}$ are irreducible for sufficiently large~$p$.
Another is that we do not know whether~$r$
is a minimal deformation of~$\rbar$ at ramified primes~$v$
for sufficiently large~$p$,  although this is predicted to hold by some generalization of Serre's conjecture.
One example to which this does apply is to the Galois representations associated to
symmetric powers of non-CM elliptic curves~$E$ over~$F$. (The conclusion of the
theorem holds for~$F$ if it holds
for any  extension~$F'/F$, and any symmetric power of~$E/F$ is  potentially
modular over some CM extension~$F/F$ in this case by~\cite{10author}.
We deduce our theorems from the modularity lifting  results of~\cite{CG} and~\cite{CG2},
  of which we assume familiarity.
  One  obstruction to directly applying the theorems of~\cite{CG}
is that the modularity results of \emph{ibid.} require further unproven assumptions, namely,
the vanishing of certain cohomology groups outside a prescribed range. The main observation
here is that vanishing in these cases may be established for all sufficiently large $p$.

For automorphic representations for~$\GL(n)$, we require the extra assumption of local--global compatibility at~$v|p$,
which is not yet known in full generality. Some results along this lines have very recently been announced in~\cite{10author}, although they 
 are not strong enough to give a completely unconditional
proof of Theorem~\ref{theorem:conditional}. One problem is that 
\cite[Theorem~4.5.1]{10author} requires the hypothesis that~$F$ contains an imaginary
quadratic field in which~$p$ splits, which has to fail for a set of primes~$p$ of positive density. It may be possible
to give an unconditional version of Theorem~\ref{theorem:conditional} under some such assumption on~$p$,
although we do not pursue this here, in part because we would still be unable
to establish for a general~$\pi$ that condition~(\ref{difficult}) holds for infinitely many~$p$.
  (The nilpotent ideal of~\cite{Scholze} and~\cite{10author} would also be an annoying
complication.)
Similarly the assumption~(\ref{convenient})  that~$\pi_v$ is a twist of Steinberg representation
can (in principle) be weakened to the weaker assumption 
that~$\pi_v$ is of the form 
$\Sp_{n_1}(\chi_1) \boxplus \Sp_{n_2}(\chi_2)
\ldots \boxplus \Sp_{n_k}(\chi_k)$
for some partition~$n= \sum_{i=1}^{k} n_i$. The main reason we do
\emph{not} do this is that it would require a more precise discussion
of local--global compatibility at~$v \nmid p$, and in particular 
a refinement of Conjecture~B of~\cite{CG}.

\subsection{Acknowledgements}

The  authors thank George Boxer for pointing out an oversight in the first version of this note.

\section{Relation with special values of periods}

The Bloch--Kato conjecture actually gives a more precise prediction of the exact order of the Selmer group
in terms of the value of the~$L$-function divided by a certain motivic period. One can think of this as two
separate conjectures. The first is to show that the normalization of~$L(1)$ by a suitable period is indeed rational.
The second is to relate the corresponding~$p$-adic valuation of this ratio to the order of a Selmer group.
Our method naturally relates the order of a Selmer group to a certain tangent space. On the other hand,
for most of the Galois representations we consider, it is not known whether there exists
a corresponding motive, and so it is not clear exactly what it means to prove rationality. There are some formulations
where one can establish certain forms of rationality (or even integrality) with respect to periods defined
in terms of automorphic integrals (see, for example,~\cite{BR}, \cite{GHL}, and also~\cite{Urban}). However, it is not clear
to the authors how these results exactly relate to the (sometimes conjectural) motivic periods. An interesting
test case is the following. Suppose that
$$\rho: G_{\Q} \rightarrow \GL_2(\C)$$
is an irreducible odd representation. According to the Artin conjecture (known in this case, see~\cite{BuzzT,BuzzWild,KhareW,KhareW2,KisinLimo}),
one knows that~$\rho$ is modular of weight one. If one chooses a prime~$p$, and supposes that~$\rho$ has a model
over~$\OL$, the finiteness of the Selmer group~$H^1_f(\Q,\ad^0(\rho) \otimes K/\OL)$ is a consequence of the finiteness
of the~$p$-class group of~$\Q(\ker(\rho))$. (The former is a quotient of the latter.) 
The methods of this paper (following~\cite{CG}) show that, at all primes~$p > 2$ such that~$\rho$ is unramified, the Selmer group~$H^1_f(\Q,\ad^0(\rho) \otimes K/\OL)$ is detected by congruences between the modular form~$f$ and other Katz modular forms of weight one which may not lift to characteristic zero. In particular, there exist such congruences if and only if~$H^1_f(\Q,\ad^0(\rho) \otimes K/\OL)$ is non-zero. However, unlike in the case of higher weight modular forms, there does not seem to be an \emph{a priori} way to relate
this to a normalization of the adjoint~$L$ function~$L(\ad^0(\rho),1)$ (which in this case is an Artin~$L$-function). The issue is that all such constructions (following Hida~\cite{HidaCong}) proceed
by understanding various parings on the Betti cohomology of  arithmetic groups in characteristic zero, whereas weight one Katz modular
forms only have an interpretation in terms of coherent cohomology. Even in cases where one does have access to Betti cohomology,
say for regular algebraic  cuspidal automorphic representations for~$\GL(n)/F$ (even for~$\GL(2)$ over imaginary quadratic fields~$F$), it is not so clear whether the cohomological pairings one can define give the ``correct'' regulators or merely the regulators up to some
finite multiple related to the torsion classes in cohomology. Since we have nothing to say about how to resolve these issues,
we follow Wittgenstein's dictum~(\cite{WT}~\S7) and say no more about them.

\section{Vanishing Theorems}

The main idea of this paper is to note that the various vanishing theorems which are required inputs for the method of~\cite{CG,CG2} may
be established at least for~$p$ sufficiently large. This is not so useful for applications to modularity --- if~$p$ is sufficiently large,
then any completion of the appropriate Hecke ring~$\T$ at a maximal ideal~$\m$ of residue characteristic~$p$ will be formally smooth of dimension one, and so
the only characteristic zero representation one can prove is modular is the representation one must assume is modular in the first place.
However, with respect to Selmer groups, this statement does have content --- it says that these representations will have
no infinitesimal deformations.

\subsection{Betti Cohomology}

Let $F$ be an imaginary CM field of degree $2d$. Let
$$l_0 := d \left( \rank(\SL_n(\C)) - \rank(\SU_n(\C)) \right) = d(n-1),$$
$$2 q_0 + l_0 = d \left( \dim(\SL_n(\C)) - \dim(\SU_n(\C)) \right) = d(n^2-1),$$
$$q_0 = \frac{d(n^2-n)}{2}.$$
Fix a tempered cuspidal automorphic representation $\pi$ for $\PGL(n)/F$ of weight zero with coefficients in $E$. 
Let~$Y = Y(K)$ be the corresponding arithmetic orbifold considered in~\S9 of~\cite{CG},
where~$K$ is chosen to be maximal at all unramified primes for~$\pi$ and Iwahori level~$\Iw_v$
for all ramified primes.
Let~$\T$ denote the (anemic) Hecke algebra defined as the~$\Z$-subalgebra of
$$\End \bigoplus_{k,m} H^k(Y(K),\Z/m \Z)$$
generated by Hecke endomorphisms~$T_{\alpha,i}$
for~$i \le n$ and~$\alpha$ which are units at primes dividing the level.
(cf~\cite[Definition~9.1]{CG}.)
For a prime $v$ of $\OL_E$,
let
$$\rbar_{v}: G_F \rightarrow \GL_n(k)$$
be the corresponding semi-simple Galois representation,
 and let $\m$ denote the corresponding maximal ideal of $\T$.

\begin{lemma} \label{lemma:vanishing} For all sufficiently large $v$, and~$\OL = \OL_{E,v}$, we have
$H^i(Y,\OL/\varpi^k)_{\m} = 0$  unless $i \in [q_0,\ldots,q_0 + l_0]$.
\end{lemma}

\begin{proof} Assume otherwise.  Pick a neat finite index subgroup~$K' \subset K$, and a corresponding
 Galois cover~$Y'  = Y(K') \rightarrow Y = Y(K)$ where~$Y'$ is now a manifold.
It follows that~$H^*(Y',\Z)$ is finitely generated, and thus~$H^*(Y,\Z[1/M])$ is also finitely generated where~$M$ denotes the
product of primes dividing~$[K:K']$.
We now assume that~$v$ has residue characteristic prime to~$[K:K']$.
Since $H^*(Y,\Z[1/M])$ is  finitely generated, the groups 
$H^*(Y,\OL) = H^*(Y,\Z[1/M]) \otimes \OL$ are torsion free and of finite rank over $\OL$ for all $i$ when $\OL$ has sufficiently large residue characteristic.
Moreover, there exist only finitely many systems of eigenvalues which occur in $H^*(Y,\R)$. 
Assuming that the result if false (and there are infinitely many~$v$), we deduce that
 there exists an eigenclass $[c]$ in $H^i(Y,\OL_E)$ with $i \notin [q_0,\ldots,q_0 + l_0]$
such that the action of $\T$ on $[c]$ has support at $\m$ for infinitely many primes $v$ of $\OL_E$.
By the Chinese remainder theorem,  the Hecke eigenvalues of $[c]$ coincide with those of $\pi$.
We now show that~$[c]$ corresponds to an automorphic form~$\Pi$ which must simultaneously be non-tempered
and yet isomorphic to~$\pi$, giving a contradiction. 
Eigenclasses in cohomology may be realized by isobaric automorphic representations
(see~\cite[Thm~2.3]{FrankeFun}).  Suppose that $[c]$
corresponds to such an automorphic representation $\Pi$. Because of the degree where $[c]$ occurs, 
we deduce (from~\cite[Ch.II, Prop~3.1]{BW} and~\cite[Lemma~3.14]{clozel-ann-arb}) that $\Pi$ is not tempered.
Yet by strong multiplicity one~\cite{J}, there is an isomorphism $\Pi \simeq \pi$.
\end{proof}

For a more detailed discussion (in a more general setting) relating
the cohomology of local systems to tempered automorphic representations, see the proof of~\cite[Thm~2.4.9]{10author}.

\begin{theorem}   \label{theorem:vanishing}  Suppose that $H^i(Y,\OL/\varpi^n)_{\m} = 0$ unless $i \in [q_0,\ldots,q_0 + l_0]$.
 Let~$Q$ be a finite collection of 
 primes~$x$ such that~$\rbar(\Frob_x)$
has distinct eigenvalues and~$N(x) \equiv 1 \mod p$.
Then 
$$H^i(Y_1(Q),\OL/\varpi^n)_{\m_{\alpha}} = 0$$
for all $i \not\in [q_0,\ldots,q_0 + l_0]$, where:
\begin{enumerate}
\item  $\alpha = \{\alpha_x\}$ is a choice of eigenvalues of $\rbar(\Frob_{x})$ for each $x$ dividing $Q$.  
\item The localization takes place with respect to the Hecke algebra~$\T_{Q}$ consisting of the Hecke
operators prime to~$p$ and prime to the level together with~$U_{x} - \alpha_x$ for all~$x \in Q$.
\end{enumerate}
In particular, the conclusions of this theorem apply for all sufficiently large~$p$.
\end{theorem}

\begin{proof}  We first note that the assumption (of absolute irreducibility) on~$\rbar$ ensures that the cohomology of the boundary
vanishes after localization at~$\m$. This is because the cohomology of the boundary may be computed inductively
from the cohomology of Levi subgroups and then of~$\GL_{n_i}$ for~$n_i < n$ (see, for example, \S3 and in particular
Prop.~3.3 of~\cite{MR3530448}), and so  the corresponding Galois
representations associated to these classes are reducible.

By Poincar\'{e} duality (and the discussion above concerning the vanishing of the boundary cohomology localized at~$\m$), it suffices to prove the result for $i < q_0$. Let $i$ be the
smallest integer for which the inequality is violated. Then, by the Hochschild--Serre spectral
sequence, we deduce that 
$$H^i(Y_0(Q),k)_{\m_{\alpha}} \ne 0.$$
As in~\S9.4 of~\cite{CG} (see also Lemma~6.25(4) of~\cite{KT}),  we deduce that $H^i(Y_0(Q),k)_{\m_{\alpha}} \simeq H^i(Y,k)_{\m}$. The
result then follows 
by Lemma~\ref{lemma:vanishing}.
\end{proof}

The modularity lifting theorems of~\cite{CG} are proved by constructing
sets of so-called ``Taylor--Wiles primes'' which have the property that imposing local
conditions at these primes annihilates (as much as possible) the dual Selmer group.
The assumption that~$\rbar_v |_{F(\zeta_p)}$ has enormous image 
implies that there exists arbitrarily many
sets~$Q$ of auxiliary Taylor--Wiles primes satisfying the hypothesis that~$\rbar(\Frob_x)$ has distinct
eigenvalues. In particular,  Theorem~\ref{theorem:vanishing} serves as a replacement
for Conjecture~B(4) of~\cite{CG}.
 (For a different  (and somewhat cleaner) treatment of Taylor--Wiles primes using the enormous image hypothesis, see~\cite{KT},
which is also used in~\cite{10author}.)

\subsection{Coherent Cohomology}

Let~$\OL$ denote the ring of integers in some finite extension of~$\Q_p$.
Let~$X$ denote a toroidal compactification of a Siegel~$3$-fold~$Y$ of level prime to~$p$ over~$\Spec \OL$, and let~$Z$ denote the minimal compactification.
Let~$\pi: \mathcal{A} \rightarrow Y$ denote the universal abelian variety, let~$\Ee = \pi_* \Omega^{1}_{\mathcal{A}/X}$,
let~$\omega = \det \Ee$, and, by abuse of notation, also let~$\omega$ denote the canonical extension of~$\omega$ to~$X$ or the corresponding ample line bundle on~$Z$.
Fix a cuspidal automorphic representation~$\pi$ for~$\GSp(4)/\Q$ corresponding to a modular abelian surface~$A$
which we assume has endomorphism ring~$\Z$ over~$\Q$, and hence to a cuspidal Siegel modular form of scalar weight~$2$. Let
$$\rbar_{p}: G_{\Q} \rightarrow \GSp_4(\F_p)$$
be the corresponding semi-simple representation for each prime~$p$. Let~$\m$
denote the corresponding maximal ideal of~$\T$. 

\begin{lemma} \label{testy} For all sufficiently large~$p$, and any set~$Q$ of auxiliary primes, we have
$$H^i(X,\omega^2_{\OL/\varpi^n})_{\m_Q} = 0$$
for~$i = 2$ and~$3$, 
where~$\m_{Q}$ denotes the maximal ideal in the Hecke algebra where operators at~$Q$ and~$p$ have been omitted. 
\end{lemma}

\begin{proof} This is a consequence of the proof of  Theorem~7.11 of~\cite{CG2}. We give a brief sketch here of the idea:
as in the proof of Lemma~\ref{lemma:vanishing}, we otherwise deduce that there exists a characteristic zero form in~$H^i(X,\omega^2_{\C})$ giving
rise to infinitely many of these classes. The representation~$\rbar_p$ will be irreducible for all sufficiently large~$p$ (because~$\End_{\Q}(A) = \Z$ --- see also the proof of Lemma~\ref{lemma:serre}). If follows that the
transfer of this form to~$\GL(4)$ must be cuspidal, and moreover
(by multiplicity one)
coincide with the transfer of the  representation coming from the
holomorphic Siegel modular form. But such a representation
only contributes to cohomology in degrees~$0$ and~$1$. (For details relating the coherent cohomology of Siegel threefolds and their
relation to automorphic representations we refer the reader back to~\cite{CG2}.)
\end{proof}

\begin{lemma} \label{lemma:reduce}  Suppose that $H^2(X,\omega^2_{\OL/\varpi^n})_{\m_Q} = 0$
for any set of auxiliary primes~$Q$, as in the statement of Lemma~\ref{testy}.
Suppose, moreover, that $\rbar_p$ is absolutely irreducible.
Then, for~$i \ge 2$,
$$H^i(X_1(Q),\omega^2_{\OL/\varpi^n})_{\m_{\alpha}}  = H^i(X_0(Q),\omega^2_{\OL/\varpi^n})_{\m_{\alpha}} = 0,$$
where $Q$ is any collection  of primes where~$\rbar(\Frob_x)$ has distinct eigenvalues, $N(x) \equiv 1 \mod p$, and
$\alpha = \{\alpha_x\}$ is any collection of eigenvalues of $\rho(\Frob_{x})$ for $x$ dividing $Q$. 
In particular, the conclusions of this theorem apply for all sufficiently large~$p$.
\end{lemma}

\begin{proof} 
In~\cite{CG2}, a somewhat elaborate version of this result  is proved in Lemmas~7.4 and~7.5.
We instead, however, use the modified treatment of Taylor--Wiles primes by Khare and Thorne (cf.\cite[Lem.\ 6.25]{KT}),
as adapted for~$\GSp(4)$ in~\S2.4 and~\S7.9 of~\cite{BCGP}, which leads to a great simplification.
By d\'{e}vissage, we can reduce to the case where
 the coefficients are a finite field~$k$. 
  To deduce vanishing for~$X_1(Q)$, we first use Serre duality to reduce the problem to vanishing 
 of~$H^i(X_1(Q),\omega(-\infty))_{\m^*_{\alpha}}$ for~$i \le 1$. 
 (Serre duality is only Hecke equivariant up to a twist by a power of the cyclotomic character, which is
 recorded by the star.)
 Equivalently, 
 it suffices to show that, if~$i$ denotes the smallest integer such that~$H^i(X_1(Q),\omega(-\infty))_{\m^*_{\alpha}}$ is non-zero,
 then~$H^i(X,\omega(-\infty))_{\m_{Q}}$ is also non-zero.
 By Hochschild--Serre applied to the map~$X_1(Q) \rightarrow X_0(Q)$,  it
 suffices to show that, if~$i$ denotes the smallest integer such that~$H^i(X_0(Q),\omega(-\infty))_{\m^*_{\alpha}}$ is non-zero,
 then~$H^i(X,\omega(-\infty))_{\m^*_{Q}}$ is also non-zero.
 The result now follows as in the proof of Proposition~7.9.8 of~\cite{BCGP},
 which identifies these groups for all~$i$ under the Taylor--Wiles hypothesis using Lemmas~2.4.36 and~2.4.37 of~\cite{BCGP}.
 \end{proof}

 \section{Proofs}
 
 Let~$R$ denote the minimal deformation ring of~$\rbar$ defined as follows:
 \begin{enumerate} 
 \item {\bf Coherent Case:\rm} $R$ is the minimal ordinary
 deformation ring denoted by~$R^{\min} = R_{\emptyset}$ in~\S4 of~\cite{CG}.
 \item {\bf Betti Case:\rm} $R$ is the minimal ordinary deformation ring
 corresponding to the following conditions:
 \begin{enumerate} 
 \item If~$v$ is a prime of bad reduction (so we
 are assuming, for a topological generator~$\sigma$ of tame inertia,
 that~$\rbar(\sigma)$ is unipotent with a single block), then  we take the
 local deformation ring to be the ring~$R^1_v$ in~\S8.5.1 of~\cite{CG}.
 Note that, if~$r(\sigma)$ on~$A^n$ has characteristic polynomial~$(X-1)^n$,
 then (given our assumption on~$\rbar$) this deformation problem
 coincides with the minimal condition in
Definition~2.4.14 of~\cite{CHT}, namely,
that the map
 $$\ker(r(\sigma) -  \mathrm{id}_n)^r \otimes_{R} k \rightarrow \ker(\rbar(\sigma) -  \mathrm{id}_n)^r$$
 is an isomorphism (equivalently, surjection) for all~$r \le n$.
 (One can see this equivalence by induction --- $\rbar(\sigma)$ has a unique eigenvector over~$k$,
 which lifts to a unique eigenvector over~$A$ whose mod-$p$ reduction is non-trivial; now take
 the representation of~$A^{n-1}$ and~$k^{n-1}$ given by quotienting out by this eigenvector.)
 \item If~$v|p$ and~$p$ is sufficiently large with respect to~$n$
 and the primes which ramify in~$F$, we take deformations
 which are Fontaine--Laffaille of weight~$[0,1,\ldots,n-1]$.
 \end{enumerate}
 \end{enumerate}

 If one has an isomorphism~$R \simeq \T$ for all sufficiently large~$p$ satisfying
 the required hypothesis, then since one also will have an isomorphism~$\T \simeq \OL$,
 this would 
  immediately imply that the tangent space to~$R$ along the projection to~$\OL$ is trivial,
 and hence the corresponding adjoint Selmer groups are trivial.
 Theorem~\ref{theorem:conditional} is now an consequence of
 Theorem~5.16 of~\cite{CG}  and Theorem~6.4 of~\cite{CG},
 where we use the fact that the corresponding local deformation
 rings are formally smooth (as follows from Corollary~2.4.3 of~\cite{CHT}
 and Lemma~2.4.19 of~\cite{CHT}),
 and where we use
 Theorem~\ref{theorem:vanishing} as a substitute for
 the vanishing assumption required in Conjecture~B of \emph{ibid}.
 Equally, Theorem~\ref{theorem:main} follows as in Theorem~1.2 of~\cite{CG2},
 where the vanishing result of Lemma~\ref{lemma:reduce} replaces the vanishing
 results of Lan--Suh~\cite{LanSuh} for other low weight local systems used in~\cite{CG2}.
 The modularity argument above requires a large image hypothesis which we assume
 in Theorem~\ref{theorem:conditional} and which we are required to prove (for sufficiently
 large~$p$) for Theorem~\ref{theorem:vanishing}. Furthermore, we must justify
 the claim in Theorem~\ref{theorem:vanishing} that the assumptions hold for a set of
 primes~$p$ of density one. Hence it remains to prove the following:
 
 \begin{lemma} \label{lemma:serre}
 Let~$A$ be a semistable abelian surface of conductor~$N$ with~$\End(A) = \Z$.
 Then:
 \begin{enumerate}
 \item  For sufficiently large~$p$,
  the residual representation~$\rbar_{p}:G_{\Q} \rightarrow \GSp_4(\F_p)$
 is surjective with minimal conductor~$N$. 
 \item For a set of primes~$p$ of density one,
 we have
 $$ \alpha \beta (\alpha^2-1)(\beta^2-1)(\alpha  - \beta)(\alpha^2 \beta^2 - 1) \not\equiv 0 \mod p.$$
 \end{enumerate}
  \end{lemma}
 
 \begin{proof}
 Since~$\End(A) = \Z$, the residual image is surjective for all sufficiently large~$p$ by~\cite{Serre},
Corollaire au Th\'{e}or\`{e}me~3. In order to ensure that the conductor of~$\rbar_p$ at a prime~$\ell$
dividing~$N$ matches that of~$A$, it suffices to take~$p$ co-prime to the  (finite) order of the component
group~$\Phi_{\mathcal{A}}$ of the N\'{e}ron model of~$A^{\vee}$ at~$\ell$.
 This proves the first claim.

For the second claim,
 let the characteristic polynomial of Frobenius (for~$p$ not dividing the discriminant on the \'{e}tale cohomology~$V_{\ell} = H^1(A,\Qbar_{\ell})$ at any prime~$\ell \ne p$) be
 $$X^4 + a_p X^3 + (2p + b_p) X^2 + p a_p X + p^2.$$
 Let the roots of this polynomial be~$\alpha$, $\beta$, $\alpha^{-1} p$ and~$\beta^{-1} p$
 respectively; by the Riemann hypothesis
 for curves (Weil bound) they are Weil numbers of absolute value~$\sqrt{p}$.
 Note that~$2 p + b_p$ is the trace of~$\Frob_p$ on~$\wedge^2 V_{\ell}$ for all
 but finitely many~$\ell$, and~$a^2_p$ is the trace of~$\Frob_p$ on~$V_{\ell} \otimes V_{\ell}$.
 We use the following lemma, which is essentially an observation of 
 Ogus (2.7.1 of~\cite{Ogus}).
 \begin{lemma} \label{lemma:ogus}
 There is no fixed linear relation between~$1$, $p$, $b_p$,  $a_p$, and~$a^2_p$ which
 can hold for a set~$p \in S$ of positive density.
 \end{lemma}
 \begin{proof}
 From such an equality, we can build two  
 finite dimensional representations~$A_{\ell}$ and~$B_{\ell}$ built out of copies
 of~$\wedge^2 V_{\ell}$, $\Q_p$, $\Q_p(1)$, $V_{\ell}$, and~$V_{\ell} \otimes V_{\ell}$
 respectively which have equal trace on~$\Frob_p$ for infinitely many~$p$.
 There must be at least one quadratic field with a positive density
 of inert primes in~$S$, twisting by this representation we arrive at a representation~$W_{\ell}$
 with a set~$S$ of positive density such that~$\Frob_p$ has trace zero for~$p \in S$.
 For sufficiently large~$\ell$, our assumptions on~$A$ implies (\cite{Serre})
 that the image of~$G_{\Q}$
 on~$V_{\ell}$ is~$\GSp_4(\Z_{\ell})$ if~$\End(A) = \Z$.
 By Cebotarev, it follows that the appropriate identity must also hold
 on an subset of these groups of positive measure. 
 Yet the distribution of the appropriate eigenvalues for~$\GSp_4(\Z_{\ell})$ does not have any atomic measure. 
In particular,
writing the eigenvalues in either case as~$x$, $y$, $\delta/x$, and~$\delta/y$, we would obtain an relation
 between the
polynomials
$$1, \quad \delta, \quad xy + \delta x/y + \delta y/x + \delta^2/xy,  \quad x+y+\delta/x + \delta/y, \quad (x+y+\delta/x + \delta/y)^2$$
 that holds on an open set (and consequently holds everywhere).
 There are no such relations by inspection.
 \end{proof}

Returning to the proof of Lemma~\ref{lemma:serre}, we deduce from the Weil bounds that $|a_p|^2 \le 16p$ and~$|b_p| \le 4p$.
 Hence, if we have any linear expression in~$a_p$, $a^2_p$, $b_p$ and~$1$ which
 is congruent to zero modulo~$p$ for a set of positive density,
 then it must also equal a constant multiple of~$p$ for a set of positive density,
 and we would obtain a contradiction by Lemma~\ref{lemma:ogus}.
 We show that this holds in each of the possible cases when our congruence
 above holds. (We take advantage of the symmetry in~$\alpha$ and~$\beta$
 and consider a reduced number of cases.)
 \begin{enumerate}
 \item Suppose that neither~$\beta$ and~$\beta^{-1} p$ are units.
Then~$b_p \equiv 0 \mod p$.
\item Suppose that~$\alpha \beta \equiv \vareps \mod p$ for some fixed~$\vareps
 \in \{\pm 1\}$. Then~$b_p \equiv \vareps \mod p$.
\item Suppose that~$\alpha = \vareps \mod p$  for some fixed~$\vareps
 \in \{\pm 1\}$. 
 Then
$b_p - \vareps a_p + 1 \equiv 0 \mod p$.
\item Suppose that~$\alpha - \beta \equiv 0 \mod p$. Then
$4 b_p - a^2_p \equiv 0 \mod p$.
 \end{enumerate}
 \end{proof}

\bibliographystyle{amsalpha}
\bibliography{CG}
 
 \end{document}